\documentclass[a4paper,10pt]{article}
\pdfoutput=1
\usepackage[utf8]{inputenc}
\usepackage{amsmath,amssymb,amsthm}
\usepackage[colorlinks,citecolor=blue]{hyperref}
\usepackage{cleveref}
\usepackage{pifont}
\usepackage{enumitem}
\usepackage{mathrsfs}
\usepackage{graphicx}
\usepackage{tikz}
\usetikzlibrary{patterns, shapes, arrows}

\newcommand{\ZZ}{\mathbb{Z}}
\newcommand{\NN}{\mathbb{N}}
\newcommand{\CC}{\mathbb{C}}
\newcommand{\QQ}{\mathbb{Q}}

\newcommand{\dy}[1]{\mathbb{#1}}

\newcommand{\eps}{\varepsilon}
\newcommand{\deri}{\mathcal{D}}

\newcommand{\trefle}{(\ding{168})}

\newtheorem{theorem}{Theorem}[section]

\newtheorem{lemma}[theorem]{Lemma} 
\newtheorem{definition}{Definition}
 
\newtheorem{remark}[theorem]{Remark}

\newcommand{\oeis}[1]{\href{https://oeis.org/#1}{{#1}}}

\newcommand{\mono}{\mathscr{W}}
\newcommand{\lcm}{\operatorname{lcm}}
\newcommand{\Cat}{\operatorname{Cat}}

\title{Posets and Fractional Calabi-Yau Categories}
\author{F. Chapoton}
\date{\today}

\begin{document}

\maketitle

\begin{verse}
\textit{To see a World in a Grain of Sand\\
And a Heaven in a Wild Flower\\
Hold Infinity in the palm of your hand\\
And Eternity in an hour}\\
\hspace{3cm}\textsc{William Blake}
\end{verse}

\begin{abstract}
  This article deals with a relationship between derived categories of
  modules over some partially ordered sets and triangulated categories
  arising from quasi-homogeneous isolated singularities. It produces
  heuristics for the existence of derived equivalences between posets,
  using the geometric category as an auxiliary intermediate. The notion
  of Weight plays a central role as a simple footprint of the derived
  categories under consideration.
\end{abstract}

\section*{Introduction}

The aim of this article is to explain a simple idea that starts from
the combinatorics of partially ordered sets (posets) and leads to
conjectures about fractional Calabi-Yau categories and their derived
equivalences.

Let us start with an infinite family of combinatorial objects, given as the
disjoint union of finite sets $P_n$ indexed for example by positive
integers. Suppose that for every $n$, the cardinality of $P_n$ can be
written under the specific shape
\begin{equation*}
    |P_n| = \frac{\prod_{i=1}^{m} (D-d_i)}{\prod_{i=1}^{m} d_i},
\end{equation*}
where $m$ is a positive integer, $d_1,\dots,d_m$ is a multi-set of positive
integers and $D$ is a positive integer, all depending on $n$ in a regular way.

Then, one can hope for the following statement \trefle:

\smallskip

\centerline{\fbox{
\begin{minipage}{0.9\textwidth}
\textit{There exists a family of partial orders $(P_n, \leq)$ such
  that, for all $n$, the derived category of $(P_n, \leq)$ is
  derived-equivalent to the fractional Calabi-Yau category associated
  with a generic isolated quasi-homogeneous singularity with variable weights
  $(d_1,\dots,d_m)$ and total weight $D$.}
\end{minipage}}}

\smallskip

In this statement, the derived category of a poset $(P, \leq)$ means
the bounded derived category of finite-dimensional modules over its
incidence algebra over a field. The meaning of the category attached
to the singularity is some kind of derived Fukaya category,
providing a categorified version of the classical Milnor
theory of isolated hypersurface singularities. Milnor theory is
briefly recalled in section \ref{singular} and the related categorical
framework is considered later in section \ref{fractionalCY}.

\smallskip

If the property \trefle\, holds, it has a much more concrete
consequence, namely the Coxeter polynomial of the poset $(P_n, \leq)$
gets identified with the characteristic polynomial of the monodromy
for the isolated singularity and can therefore be computed by just
knowing the numbers $d_1,\dots,d_m$ and $D$, by results of Milnor and
Orlik explained in section \ref{singular}. The data of $d_1,\dots,d_m$
and $D$, satisfying the appropriate conditions, will be called a
\textit{Weight}\footnote{We write Weight with a capital
  letter to make it more clear that it stands for a precise technical
  definition, see \Cref{define_Weight}.}.

As the Coxeter polynomial is also easy to compute directly for any
given poset (using the Coxeter matrix, see App. \ref{cox}), this gives a criterion to check if a family of partial
orders on $P_n$ could have the expected property. When the first Coxeter polynomials for small $n$ are as expected, this gives a
strong evidence for the statement \trefle\, above. In this case, let us say that the
\textit{Coxeter criterion} holds for this family of posets.

\smallskip

Here is a schematic description of the situation.

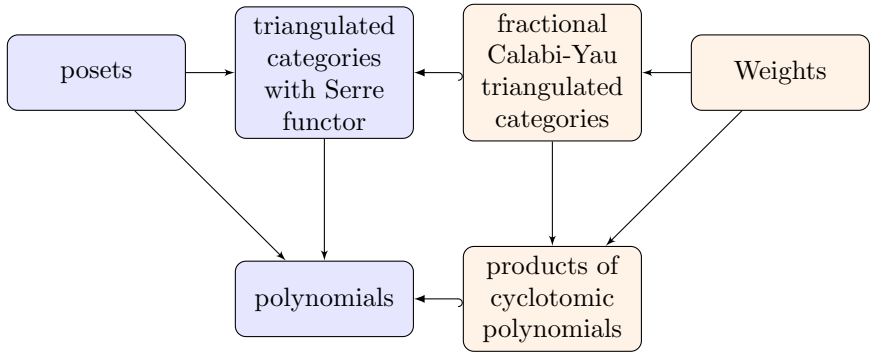
\tikzstyle{block} = [rectangle, draw, fill=blue!10, 
    text width=6em, text centered, minimum height=1cm, rounded corners]
\tikzstyle{block2} = [rectangle, draw, fill=orange!10, 
    text width=6em, text centered, minimum height=1cm, rounded corners]
\tikzstyle{line} = [draw, -latex']

\begin{center}
\begin{tikzpicture}[node distance = 3cm, auto]
    \node [block] (poset) {posets};
    \node [block, right of=poset] (tri) {triangulated categories with Serre functor};
    \node [block2, right of=tri] (fcytri) {fractional Calabi-Yau triangulated categories};
    \node [block2, right of=fcytri] (weight) {Weights};
    \node [block, below of=tri] (poly) {polynomials};
    \node [block2, below of=fcytri] (cyclopoly) {products of cyclotomic polynomials};
    \path [line] (poset) -- (tri);
    \path [line] (weight) -- (fcytri);
    \path [draw, right hook-latex] (fcytri) -- (tri);
    \path [line] (poset) -- (poly);
    \path [line] (weight) -- (cyclopoly);
    \path [draw, right hook-latex] (cyclopoly) --(poly);
    \path [line] (tri) --(poly);
    \path [line] (fcytri) --(cyclopoly);
\end{tikzpicture}
\end{center}

The top-left arrow sends a poset to the derived category of modules over its
incidence algebra over a field. This is a many-to-one application, as
one can easily find examples of derived-equivalent but not isomorphic
posets. The top-right arrow is the Fukaya-Seidel categorification of the
Milnor fibre construction. The
arrows from triangulated categories to polynomials are given by the
characteristic polynomial of the Auslander-Reiten functor (which is a
shifted version of the Serre functor). As said above, there are direct
constructions of these polynomials from posets and from Weights, which give the diagonal arrows.

\smallskip

This diagram is compatible with the natural monoidal structures,
namely cartesian product of posets, tensor product of triangulated
categories, and the monoid structure on Weights introduced in section
\ref{monoid}. For the bottom line, one can use a tensor product of
polynomials, but we will not need that.

\smallskip

So the Coxeter criterion means that we have a sequence of polynomials
coming from the top-left that can be identified with a sequence of
polynomials coming from the top-right. The main idea is to take this
equality as a rather strong hint that the triangulated categories should be
themselves equivalent.

\smallskip

As we will see in the examples below, partial orders on a given
combinatorial family that satisfy the Coxeter criterion are not
necessarily unique. There can very well be several distinct families
of posets with the same cardinalities, all having their derived
categories derived-equivalent to the same singularity category. In
this case, the implied derived equivalences between the different
posets of the same cardinality can sometimes be proved by other means.

\smallskip

Given a family of combinatorial objects, finding some correct partial
orders is not easy in general. Sometimes the most natural partial
orders all fail to satisfy the Coxeter criterion. One can then look
for more subtle partial orders on the same combinatorial objects, or maybe
on other combinatorial objects counted by the same sequence of
numbers.

\smallskip

Another implication is that one can use factorisations in the monoid
of Weights to propose conjectural derived equivalences. If a Weight is
associated with a poset $P$, and factors into simpler Weights
associated with smaller posets $Q_1,\dots,Q_k$, then one should expect
a derived equivalence between $P$ and the cartesian product of posets
$Q_1 \times \dots \times Q_k$. A very simple case is given by the Dynkin
quivers $\dy{D}_4$ and $\dy{A}_2 \times \dy{A}_2$.


\smallskip

Acknowledgements: Thanks to Guo-Niu Han for the proof of Lemma
\ref{lemme_ASM}. Thanks to Pierre Baumann and Sefi Ladkani for
interesting discussions and useful comments. Thanks to the referee of
the first version for convincing explanations about the existence of
appropriate Fukaya categories. Thanks to Alexandru Oancea and Yanki
Lekili for very helpful guidance in the symplectic side of the matter.

This research has been supported by the ANR project Charms (ANR-19-CE40-0017).

\section{Quasi-homogeneous isolated singularities}

\label{singular}

The theory of singularities of algebraic functions from $\CC^m$ to
$\CC$ is a very classical topic, with a vast literature. One could
cite for instance the famous classification by Arnold of rigid
isolated singularities by the Dynkin diagrams of type $\dy{A}\dy{D}\dy{E}$. Even more
well-studied is the case of singularities of quasi-homogeneous
algebraic functions, where each coordinate on $\CC^m$ is given a
specific weight and the function is assumed to be homogeneous for the
total degree with respect to these weights.

Let us sketch briefly the celebrated construction of Milnor for
isolated singularities. For more details, the reader may consult
\cite{milnor_livre, dimca}. Let $f$ be a quasi-homogeneous polynomial
function from $\CC^m$ to $\CC$. Assume that $f$ has an isolated
critical point at $0 \in \CC^m$ above $0 \in \CC$. Milnor has shown that
over a sufficiently small circle $S_\eps$ around $0 \in \CC$, all the fibres
of $f$ (intersected with a small sphere around $0 \in \CC^m$) are
smooth and diffeomorphic, with the homotopy type of a bouquet of $\mu_f$
spheres of dimension $m-1$. The fibres have therefore only one
interesting homology group $H_{m-1}$, of dimension $\mu_f$. This locally-trivial
fibration over the circle $S_\eps$ is called the \textit{Milnor
  fibration} of $f$ and $\mu_f$ is called the \textit{Milnor number} of the
singularity.

By turning once over the circle $S_\eps$ and following the cycles
using local triviality of the Milnor fibration, one gets a linear
endomorphism of the homology group $H_{m-1}(f^{-1}(\eps))$. This is
called the \textit{monodromy} of the singularity.

In the case of a quasi-homogeneous polynomial $f$ with an isolated
singularity, Milnor and Orlik \cite{milnor_orlik} have given an
explicit formula for the characteristic polynomial of the monodromy
(or rather for its roots) depending only on the degrees
${d_1,\dots,d_m}$ of the variables and the total degree $D$ of $f$. As
a special case of this formula, the Milnor number is given by
\begin{equation}
  \label{mu_formula}
  \mu_f = \frac{\prod_{i=1}^{m} (D-d_i)}{\prod_{i=1}^{m} d_i}.
\end{equation}

Let us now present their formula briefly. The following description is
a streamlined presentation, with slightly modified notations, of Milnor and
Orlik result in \cite[\S 3]{milnor_orlik}, see also \cite[\S 9]{milnor_livre}.
Let $u_i$ be the numerator
of $D/d_i$ as a reduced fraction, for $i=1,\dots,m$. For each divisor
$j$ of $D$, define
\begin{equation}
  \chi_j = \prod_{\stackrel{i=1}{u_i | j}}^{m} \frac{d_i-D}{d_i}.
\end{equation}
Then for each divisor $j$ of $D$, there is a relative integer $s_j$ such that
\begin{equation}
  j s_j = \sum_{d | j} \mu_{j/d} \, \chi_d,
\end{equation}
where $\mu$ is (just here) the standard number-theoretic Möbius function. The
characteristic polynomial of the monodromy is then
\begin{equation}
  \left(\prod_{j|D} (t^j-1)^{s_j}\right)^{(-1)^m}.
\end{equation}

Consider for example the degrees $(d_1,d_2,d_3)=(2,3,4)$ and $D =
10$. Then $(u_1,u_2,u_3) = (5,10,5)$. One computes that
$(\chi_1,\chi_2,\chi_5,\chi_{10})=(1,1,6,-14)$ and it follows that
$(s_1,s_2,s_5,s_{10})=(1,0,1,-2)$. As a product of cyclotomic
polynomials, the characteristic polynomial is therefore
$\Phi_2^2 \Phi_5 \Phi_{10}^2$.  

\subsection{Hodge structure and $q$-Milnor number}

This information about the characteristic polynomial can been refined
as follows.

Let us consider the formula
\begin{equation}
  \label{q_milnor_formula}
  \frac{\prod_{i=1}^{m} [D-d_i]_q}{\prod_{i=1}^{m} [d_i]_q},
\end{equation}
where $[d]_q = (q^d-1)/(q-1)$ is the $q$-analogue of an integer $d$.

For the degrees $(d_1,\dots,d_m)$ and $D$ of an isolated
quasi-homogeneous singularity, it is known that
\eqref{q_milnor_formula} is a polynomial in $\NN[q]$, whose
coefficients are the dimensions of the homogeneous components of the
Jacobian algebra of the singularity. This is a classical statement,
see \cite{arnold_gusein}, \cite[Theorem 6.4]{hertling_mase} or
\cite[5.11]{steenbrink_oslo}.

In this situation, the polynomial \eqref{q_milnor_formula} is also
closely related to the eigenvalues of the monodromy. As shown by
J. Steenbrink \cite{steenbrink_oslo, steenbrink_2022}, the homology group
$H_{m-1}$ carries a mixed Hodge structure, compatible with the
monodromy. In the case of quasi-homogeneous isolated singularities,
the dimensions of the successive quotients of the Hodge filtration can
be encoded by the coefficients of a polynomial in $q$, which according
to \cite[5.11]{steenbrink_oslo} turns out to be \eqref{q_milnor_formula}.

The polynomial \eqref{q_milnor_formula} in fact also contains complete
information on the multiplicities of the eigenvalues of the monodromy,
hence gives an alternative way to access them, different from the
Milnor-Orlik method recalled above. It suffices to look at
the coefficients of the unique polynomial representative of
\eqref{q_milnor_formula} modulo $q^D-1$ with degree at most $D-1$. The
coefficients, appropriately centered, are then the multiplicities of the $D$-th roots of unity
as eigenvalues of the monodromy.

For a more precise account of these aspects, see \cite[\S 5, \S 6]{hertling_mase}.

\smallskip

One can note that setting $q=1$ in \eqref{q_milnor_formula} recovers
the expression \eqref{mu_formula} for the Milnor number $\mu_f$. The
expression \eqref{q_milnor_formula} will be called the
\textit{$q$-Milnor number} of the singularity.

\smallskip

In the example of degrees $(2,3,4)$ and total degree $10$, one finds
\begin{equation*}
  q^{12} + q^{10} + q^{9} + 2q^{8} + q^{7} + 2q^{6} + q^{5} + 2q^{4} + q^{3} + q^{2} + 1,
\end{equation*}
which reduces modulo $q^{10} -1$ to
\begin{equation*}
  q^{9} + 2q^{8} + q^{7} + 2q^{6} + q^{5} + 2q^{4} + q^{3} + 2 q^{2} + 2.
\end{equation*}
This is compatible with the characteristic polynomial, which is
\begin{equation*}
(t^5 + 1)^2 (t^4 + t^3 + t^2 + t + 1),
\end{equation*}
as seen at the end of the previous section.

\section{Fractional Calabi-Yau categories}

\label{fractionalCY}

Fractional Calabi-Yau categories were introduced by M. Kontsevich around
1998 \cite{kont1998} as a natural generalisation of Calabi-Yau
categories, themselves motivated by the properties of coherent sheaves
on Calabi-Yau manifolds. Fractional Calabi-Yau categories sometimes
appear in the semi-orthogonal decompositions of bounded derived
categories of coherent sheaves on algebraic varieties, and in
particular Fano varieties, see for example \cite{kuznetsov}.

\smallskip

Recall that a \textit{Serre functor} in a triangulated category
$\mathcal{T}$ is an auto-equivalence $S$ of $\mathcal{T}$ such that
there is a bi-natural isomorphism
\begin{equation*}
  Hom(X,Y)^* \simeq Hom(Y, SX),
\end{equation*}
where $*$ is the linear dual over the ground field. For more on this
notion, we refer to \cite{bondal_kapranov, keller_CY}. The existence
of a Serre functor $S$ on a triangulated category is equivalent to the
existence of an Auslander-Reiten translation functor $\tau$. These two
functors are unique up to isomorphism and related by $S =
\tau[1]$, where $[1]$ is the shift functor. They both exist for example for the bounded derived
categories of modules over a finite dimensional algebra of finite
global dimension over a field, see \cite[\S 3.1]{keller_CY}. This includes
incidence algebras of finite posets over a field.

\smallskip

A triangulated category $\mathcal{T}$ is a \textit{fractional Calabi-Yau
  category} if it has a Serre functor $S$ and there exist integers $p$
and $q$ such that $S^q \simeq [p]$ as functors. Here $[p]$ is the
$p$-th power of the shift functor. In this case, the Calabi-Yau
dimension is the pair $(p,q)$, often denoted $p/q$ by a common abuse
of notation.

\subsection{Fukaya-Seidel categories for singularities}

Let us keep the same notations as in the previous section.

Attached to each quasi-homogeneous isolated singularity $f$, there is
a triangulated category $\deri_f$ which is a categorification of the
Milnor geometric theory described in section \ref{singular}. This
category has been defined by Seidel \cite{seidel, seidel_more} in the
context of symplectic geometry. It is obtained as the derived category
(or homology category) of a $A_\infty$-category of Fukaya type, in a
directed version, starting from the Milnor fibration and using a morsification.

\begin{remark}
  \textbf{Word of warning:} The author, lacking expertise in the field
  of symplectic geometry, failed to find good references for the three
  plausible statements below, and urge the reader to take them with
  appropriate caution. Most of the articles on closely related topics
  seem to refer directly to the original book by Seidel
  \cite{seidel_more} for the definition of the Fukaya-Seidel
  categories. Some relevant articles are
  \cite{futaki11,futaki13,haber20,haber22,favero23,auroux}.
\end{remark}

(S0) The category $\deri_f$ has the property that the Grothendieck
group $K_0(\deri_f)$ is identified with the homology group $H_{m-1}$,
in such a way that the Auslander-Reiten functor on $\deri_f$ induces a
linear endomorphism on $K_0(\deri_f)$ which is identified with the
monodromy (up to an appropriate shift). This statement seems to be
folkore in the domain.

(S1) This category $\deri_f$ is fractional Calabi-Yau,
with Calabi-Yau dimension $(C, D)$ where
\begin{equation}
  \label{CY_dimension}
  C = \sum_{i=1}^{m} (D - 2 d_i).
\end{equation}

(S2) The construction $f \mapsto \deri_f$ is multiplicative, sending the
Thom-Sebastiani sum of singularities to the tensor product of
triangulated categories. This was explicitely stated as Conjecture 1.3
in \cite{auroux} and is apparently still open.

\subsection{About mirror symmetry and $B$-model}

Mirror symmetry suggests the existence of mirror singularities for
isolated quasi-homogeneous singularities in general. This is known
only in some specific cases and has been in particular much studied in
the case of invertible polynomials, as considered by Berglund and
Hübsch in \cite{berglund-h} and by Kreuzer and Skarke in
\cite{kreuzer_skarke}. Then more algebraic methods are available to
define and study the same categories, for instance homological matrix
factorisations. For a more precise view on this, the reader may see
\cite{kajiura_saito}, \cite[\S 4]{ebeling} and \cite{ebeling_ploog}.

\smallskip

In the language of theoretical physics (and with all the necessary
caution), the category $\deri_f$ has something to do with the
$A$-model for the Landau-Ginzburg potential $f$, and it should be
related to $A$-branes of this model. The mirror symmetry is supposed
to identify these $A$-branes with $B$-branes on the mirror manifold,
which seem to be better understood in mathematical terms.





\section{The monoid of Weights}

\label{monoid}

Let us define in this section a monoid $\mono$ whose elements will be
called \textit{Weights}. Our notion of Weight is closely related to
what is called a \textit{weight system} in singularity theory, see for
example \cite{hertling_mase}.

\begin{definition}
  \label{define_Weight}
  A Weight is a pair $((d_1, d_2, \dots, d_m), D)$ where
  $d_1,d_2,\dots,d_m$ and $D$ are positive integers such that the
  formula
  \begin{equation}
    \label{q_formula}
    \frac{\prod_{i=1}^{m} [D-d_i]_q}{\prod_{i=1}^{m} [d_i]_q}
  \end{equation}
  defines a polynomial in $\NN[q]$. This Weight will be denoted
  $(d_1,\dots,d_m ;D)$.

  The order of the $d_i$ is irrelevant. Weights that only differ by
  multiplying all $d_i$ and $D$ by a common positive integer $N$ are
  considered to be the same.
\end{definition}
Note that $m=0$ is allowed, as the formula is then the empty product.

One will always assume that $d_i < D - d_i$ for all $i$, as
factors where $2 d_i = D$ do not contribute to the product \eqref{q_formula}.

Dividing every $d_i$ and $D$ by their greatest common divisor and then
sorting the $d_i$ in increasing order gives a unique canonical representative.

For example, in the case $(2, 3 ; 8)$, one finds the fraction
\begin{equation}
  \frac{[6]_q [5]_q}{[2]_q [3]_q} = q^{6} + q^{4} + q^{3} + q^{2} + 1,
\end{equation}
so that this is indeed a Weight.

The value at $q=1$ of the formula \eqref{q_formula} for a Weight $\alpha$ will
be called the \textit{Milnor number} $\mu_\alpha$ of the Weight. For
example, the Milnor number of $(2, 3 ; 8)$ is $5$.

The expression \eqref{q_formula} will be called the \textit{$q$-Milnor
  number} of the Weight. Note that it depends on the choice of a
representative, but only up to substitution of $q$ by some power of
$q$.

\subsection{Variations}

One can make several variants of the definition above, some weaker and
some stronger.

When the condition on the pair $((d_1, d_2, \dots, d_m), D)$ in this
definition is weakened to require only that the value of
\eqref{q_formula} at $q=1$ is an integer, this will be called a \textit{weak Weight}.

One could similarly require that the quotient should be a polynomial in
$\ZZ[q]$ with value at $q=1$ in $\NN$. We will not use this
intermediate notion. It is not clear if one can find something like
this which is not a Weight.

For a given Weight, one can consider a generic polynomial of degree
$D$ in variables $x_1,\dots,x_m$ of degrees $d_1,\dots,d_m$. In order
for such a polynomial to define an isolated hypersurface singularity,
a stronger condition must be imposed on the Weight, which can be found
for example in \cite[\S 2]{hertling_kubel}. Examples of Weights not
satisfying this stronger condition, such as $(16, 18, 21, 55; 165)$,
are displayed in \cite[Table 1]{hertling_mase}.

\subsection{Product}

Let $\mono$ be the set of all Weights. The set $\mono$ can be endowed with the following binary operation. In terms of singularity theory, this
corresponds to the Thom-Sebastiani direct sum of hypersurface
singularities.

Let
$\alpha = (a_1,\dots,a_m;A)$ and $\beta = (b_1,\dots,b_m;B)$ be two Weights. Then one defines
\begin{equation}
  \alpha \times \beta = (B a_1, B a_2, \dots, B a_n, A b_1, \dots, A b_n ; AB),
\end{equation}
One can check that this is indeed a Weight. This could also be defined as
\begin{equation*}
  \alpha \times \beta = (B' a_1, B' a_2, \dots, B' a_n, A' b_1, \dots, A' b_n ; \lcm(A, B)),
\end{equation*}
where $A' = A/\gcd(A,B)$ and $B'=B/\gcd(A,B)$, which is a simpler
representative of the same Weight. One could also define the same
operation as disjoint union of the $a_i$ and $b_i$ by assuming without
loss of generality that $A = B$.

This defines a commutative and associative product $\times$ on the set
$\mono$, with unit the empty Weight $(\emptyset ; 1)$.

For example $(3,5;20) \times (1;5) = (3,4,5;20)$.

\smallskip

There are a few interesting morphisms from $\mono$ to other monoids.

The formula \eqref{q_formula} evaluated at $q=t^{1/D}$ defines a
morphism to the multiplicative monoid of
Puiseux polynomials in $t$.

Similarly, the evaluation of \eqref{q_formula} at $q=1$ defines a
morphism to the multiplicative monoid
$\NN$. This is just the Milnor number.

From \eqref{CY_dimension}, one obtains the formula
\begin{equation}
  \label{central_charge}
  \frac{\sum_{i=1}^{m} (D - 2 d_i)}{D}
\end{equation}
for the Calabi-Yau dimension seen as a positive rational number. This
defines a morphism to the additive monoid $\QQ_{>0}$. This is clear
when seeing $\times$ as concatenation of Weights sharing the same
$D$. This quantity could be called the \textit{central charge} of the
Weight\footnote{It is one third of the central charge appearing in the
  related $N=(2,2)$ super-conformal field theory.}.

For example, the central charge of $(2,3,5,5;15)$ is $2$.

\subsection{Factorisation and prime Weights}

Let us say that a Weight is \textit{prime} if it is not the product of
several strictly smaller Weights, namely with smaller number of degrees.

In order to check that a Weight $(d_1,\dots,d_m;D)$ is prime or not,
one needs to look for non-empty subsets of the $d_i$ that define a Weight and
whose non-empty complement also defines a Weight (keeping the same $D$). This
can be done using the expression of $q$-integers as products of
cyclotomic polynomials $\Phi_d$. In small cases, one can easily check
in this way that some Weight is prime.

As a simple example, let us prove that the Weight $(3,4,5,6;15)$ is
prime. The factors in the $q$-Milnor number are, after simplification,
\begin{equation}
  \frac{\Phi_2 \Phi_4 \Phi_6 \Phi_{12}}{1} \times \frac{\Phi_{11}}{\Phi_2 \Phi_4} \times \frac{\Phi_2 \Phi_{10}}{1} \times \frac{\Phi_9}{\Phi_2 \Phi_6}.
\end{equation}
Because of $\Phi_6$ in its denominator, the fourth term must be
grouped with the first one. Then because of $\Phi_4$ in its
denominator, the second term must also be grouped with the first
one. Then there remains a $\Phi_2$ in the denominator of the result
which is forced to also be grouped with the third term.

Consider now the Weight $(2, 4, 6, 7 ; 18)$ with
Milnor number $88$. One can check that it can be written both as
\begin{equation}
  (1 ; 9) \times (4, 6, 7 ; 18) \quad\text{and as }\quad
  (1 ; 3) \times (2, 4, 7 ; 18),
\end{equation}
where in both factorisations all factors are prime. Something similar
happens for the Weight $(3, 4, 7, 10 ; 24)$. It follows that in the
monoid $\mono$ there is no unique factorisation in prime elements.


\section{The Catalan family}

\label{catalan}

In this section and the following ones, we consider several examples
of families of Weights, starting from a simple case.

\smallskip

The Catalan numbers are defined by the formula
\begin{equation}
  \label{cat}
  c_n = \frac{1}{n+1}\binom{2n}{n},
\end{equation}
which can be written as
\begin{equation*}
  c_n = \frac{2n \dots n + 2}{2 \dots n}
\end{equation*}
for $n \geq 1$. This comes from the Weight $(2,3,\dots,n ; 2n+2)$. In
this case, it is known that the $q$-Milnor number is a polynomial in
$q$ with positive coefficients. Indeed, this polynomial is enumerating
Dyck paths according to the major index, see \cite[St000027]{FindStat}
and \cite{macmahon}.

For small $n$, the Weights in this family are
$\dy{A}_1,\dy{A}_2,\dy{D}_5,S_{1,0},\dy{D}_7 \times \dy{E}_6$, using
factorisation in $\mono$ and the notations in the tables of section \ref{tables}.

The next Weight in this family can still be described using the notion
of invertible polynomials, as considered by in \cite{berglund-h} and
classified in \cite{kreuzer_skarke}: it is the Weight corresponding to
the chain type of parameter $(7, 2, 2, 2, 3)$.

In general, the Weights in the Catalan family cannot be realized by
invertible polynomials in the sense of Berglund-Hübsch. The first
impossible case happens for $n=10$, with Milnor number $16796$.

\smallskip

It turns out that there are at least two different families of posets
that seem to satisfy the Coxeter criterion.

The first family is made of the Tamari lattices $T_n$, introduced by
Tamari in \cite{tamari0}. They have been studied a lot since then,
in particular as a special case of the Cambrian lattices in the theory
of cluster algebras, see for instance \cite{tamari_festschrift}. The
underlying set is the set of planar rooted binary trees with $n$ inner
vertices and $n+1$ leaves, endowed with the partial order whose
covering relations are rotations. The cardinality of $T_n$ is known to
be the Catalan number $c_{n}$.

In the case of the Tamari lattices, the Coxeter polynomial has been
computed by operadic methods in \cite{chapoton_ternaire}. It is
therefore possible to check the Coxeter criterion for large
values of $n$. In principle, one could hope to prove that this general
formula coincides with the formula obtained from the sequence of
Weights, although this has never been done to our knowledge.

Moreover, B. Rognerud has proved in \cite{rognerud_CY} that the derived
category of $T_n$ is indeed fractional Calabi-Yau, of the expected
dimension $(n(n-1),2n+2)$.

\smallskip

The second family of posets is even simpler. The underlying set is
the set $D_n$ of Dyck paths of size $n$, which are lattice paths of
length $2n$ using steps $(+1,+1)$ and $(+1,-1)$, starting from
$(0,0)$, ending at $(2n,0)$ and never going strictly below the
horizontal axis. The number of such paths is known to be $c_n$
too. The partial order on $D_n$ is defined by one path being always
weakly below another path. This defines a distributive lattice.
In this case, no general formula is known for the Coxeter polynomials,
but one can check by computer that they coincide with those of the
Tamari lattices for $n \leq 9$.

All this strongly suggests that the posets $T_n$ and $D_n$ are derived-equivalent, and that both are derived-equivalent to the same
triangulated category of geometric origin associated with an isolated
singularity. For this reason, this derived equivalence has been stated
as a conjecture in \cite{chapoton_CY}.

Recently, some intermediate lattices (named the alt-Tamari lattices)
have been introduced in \cite{cheneviere}, that generalize the
previous two families and apparently share the same Coxeter
polynomials. In a manuscript in preparation with S. Ladkani
\cite{chapoton_ladkani}, we plan to establish the derived equivalences
among these intermediate lattices and in particular solve the above
conjecture on the derived equivalence of the lattices $T_n$ and $D_n$.


\medskip

Let us now briefly talk about closely related posets where the Coxeter
criterion seems to hold, for other sequences of cardinalities given by
similar formulas. In each case, one can guess the Weights from the formula.

First there are some posets enumerated by the Fuss-Catalan numbers,
namely the $m$-Tamari lattices introduced by F. Bergeron and
L.-F. Préville-Ratelle \cite{bergeron_preville} and also the simpler
posets of $m$-Dyck paths under the relation of being weakly below. The
larger family of rational Tamari lattices can also be considered, as
they are counted by a similar formula.

Second, there are the Cambrian lattices associated with a finite
Coxeter group $W$, all enumerated by the ``Coxeter-Catalan number'' for
$W$. In this case, the derived equivalence between these posets, for a
given $W$ and all choices of Coxeter element, has been proved by
Ladkani in \cite{ladkani_cluster}.

Third, there are the partial order on tilting modules (or positive
clusters) for a Weyl group $W$, enumerated by the ``positive
Coxeter-Catalan numbers'' for $W$. In this case too, the derived
equivalence between these posets, for a given $W$ and all choices of
Coxeter element, has been proved by Ladkani in \cite{ladkani_tilting}.

\section{Alternating Sign Matrices}

Alternating sign matrices are combinatorial objects generalizing
permutation matrices, that appeared in the Dodgson condensation
algorithm for computing the determinant. The number of alternating
sign matrices of size $n$ is given by the famous formula
\begin{equation}
  \label{formule_ASM}
  \prod_{k=0}^{n-1} \frac{(3k+1)!}{(n+k)!},
\end{equation}
which was conjectured by Mills, Robbins and Rumsey
\cite{mills_et_alii}, first proved by Zeilberger \cite{zeilberger} and
proved again by Kuperberg \cite{kuperberg} using methods of
statistical mechanics. It is an open problem to find an explicit
bijection between alternating sign matrices and totally symmetric
self-complementary plane partitions, which were enumerated by the same
formula by Andrews in \cite{andrews}. For a detailed account of the
full story, see \cite{bressoud_propp, bressoud_livre}.

\begin{lemma}
  \label{lemme_ASM}
  The formula \eqref{formule_ASM} is the Milnor number
  of the weak Weight
  \begin{equation}
    (\{3k+2,\dots,n+k\}_{0 \leq k \leq n-k-2}  ; 3n).
  \end{equation}
\end{lemma}
\begin{proof}
  Formula \eqref{formule_ASM} is the quotient of a product of
  factorials by a product of factorials. Consider the $k$-th and
  $(n-1-k)$-th factorials in the numerator, together with the $k$-th and
  $(n-1-k)$-th factorials in the denominator. This gives the quotient
  \begin{equation*}
    \frac{(3k+1)!(3n-3k-2)!}{(n+k)!(2n-k-1)!}.
  \end{equation*}
  Assuming that $k < n-1-k$, this is
  \begin{equation*}
    \frac{(2n-k)\dots(3n-3k-2)}{(3k+2)\dots(n+k)}
  \end{equation*}
  which can be written as the product
  \begin{equation*}
    \prod_i \frac{D-d_i}{d_i}
  \end{equation*}
  where $(d_1,\dots,d_m)=(3k+2,\dots,n+k)$ and $D = 3n$. When $n$ is
  odd and $2k = n-1$, the middle terms in the numerator and the
  denominator of \eqref{formule_ASM} are both $(3k+1)!$ hence can be
  neglected. So the full expression is indeed associated with this
  weak Weight.
\end{proof}

In fact, this weak Weight should be a Weight, and its $q$-Milnor number
should be a polynomial with positive integer coefficients. This can be
checked for $n \leq 30$, but there is no known combinatorial
statistics to explain this property.

For small $n$, the Weights in this family are
$\dy{A}_1,\dy{A}_2,\dy{E}_7,\dy{D}_7 \times \dy{E}_6$.

The expected Calabi-Yau dimension is then $(2 \binom{n + 1}{3}, 3n)$.






\smallskip

There are several natural partial orders on objects enumerated by
formula \eqref{formule_ASM}. The first one is the enveloping lattice
(or Dedekind-MacNeille completion) of the Bruhat order on the
symmetric group $\mathcal{S}_n$, as proved in
\cite{lascoux_schutz}. These posets do not meet the Coxeter
criterion. 

J. Striker has introduced in \cite[\S 5]{striker}, as part of a
more general construction involving a choice among colours, two
families of distributive lattices having \eqref{formule_ASM} as
cardinalities.

The first family (for the colours {\bf b}lue, {\bf y}ellow, {\bf
  o}range and {\bf g}reen in the terminology of \cite{striker}) has
elements in bijection with the alternating sign matrices, and is in
fact isomorphic to the enveloping lattice above.

The second family (for the colours {\bf r}ed, {\bf y}ellow, {\bf
  o}range and {\bf g}reen) has elements in bijection with the totally
symmetric self-complementary plane partitions. Experimentally, the
posets in this family do have the correct Coxeter polynomial for
$n \leq 5$, hence satisfy the Coxeter criterion. So
conjecturally, all these posets should be fractional Calabi-Yau. 



As a side remark, one can note that the poset of size $42$ in this
family seems to be derived-equivalent to the posets of size $42$ in
the Catalan family, and they share the same Weight $\dy{D}_7 \times \dy{E}_6$.

\smallskip

\begin{remark}
  The same idea can be applied to other symmetry classes of plane
  partitions, which are often enumerated by a closed formula involving
  a product. This includes the full set of plane partitions inside an
  $a \times b \times c$ box and the famous formula of MacMahon. This
  also seems to work for totally symmetric plane partitions
  (\oeis{A005157}) and cyclically symmetric plane partitions
  (\oeis{A006366}). In the totally symmetric case, one observes amusing
  coincidences:
  \begin{itemize}
  \item in cardinality $66$, for the
    Weight $\dy{A}_{11} \times \dy{E}_{6}$ with the Weight for the poset of tilting modules of type
    $\dy{F}_4$.
  \item in cardinality $2431$, for the
    Weight $\dy{A}_{17} \times Z_{13} \times Q_{11}$ with the Weight for the poset of tilting modules of type
    $\dy{E}_7$.
  \end{itemize}
\end{remark}


\section{The West family}

In his famous article \cite{west}, West introduced the notion of
\textit{2-stack sortable permutations} and conjectured that the number
of such permutations on $n$ letters is given by the formula
\begin{equation}
  \label{twostack}
  2 \frac{(3 n)!}{ (2n+1)! (n+1)!}.
\end{equation}
This was first proved by Zeilberger in \cite{zeilberger_stack}.

The formula \eqref{twostack} can be written as
\begin{equation*}
  \frac{(3n)\dots(2n+2)}{(3)\dots(n+1)},
\end{equation*}
which comes from the weak Weight
\begin{equation}
  \label{west_weight}
  (3,\dots,n+1;3n+3).
\end{equation}

Here again, it is not clear that the $q$-Milnor number is a polynomial
in $q$ with positive coefficients. One can check by computer that this
is the case for $n \leq 50$.  Assuming that this always holds and
therefore that \eqref{west_weight} defines a Weight, one can look for
posets satisfying the Coxeter criterion.

For small $n$, the Weights in this family are
$\dy{A}_1,\dy{A}_2,\dy{E}_6,\dy{A}_2 \times Z_{11}, \dy{E}_7 \times Z_{13}$.

The next Weight in this family can still be described using the notion
of invertible polynomials, as considered in \cite{berglund-h} and
classified in \cite{kreuzer_skarke}: it is product of $\dy{A}_2$ by
the Weight corresponding to the chain type of parameter $(7, 3, 3, 4)$.

In general, the Weights in the West family cannot be realized by
invertible polynomials in the sense of Berglund-Hübsch. The first
impossible case happens for $n=8$, with Milnor number $9614$.

\smallskip

Besides 2-stack sortable permutations, there are several other families of combinatorial objects with the
same cardinality: left-ternary-trees \cite{left_ternary}, fighting fishes
\cite{fighting_fish}, non-separable planar maps \cite{brown_tutte} and synchronized Tamari
intervals \cite{viennot_lfpr}. The last family is in bijection with the maximal cells in the
diagonal of the associahedra \cite{masuda, combe}.

On these combinatorial objects, one can find several partial
orders. One possibility is by restriction of partial orders on
permutations (weak order, Bruhat order, \textit{etc.}) to the subset of 2-stack
sortable permutations. Another is to use the geometry of the diagonal
of the associahedra, which is naturally oriented.

These tentatives have met no success so far, always failing the
Coxeter criterion as soon as $n$ is not very small. So, the question
remains whether there does exist such a family of posets. One can
even hope for the existence of posets whose Hasse diagrams would be
the oriented 1-skeletons of a sequence of simple polytopes having
their $h$-vectors given by \oeis{A082680}.

\section{The Tamari-intervals family}

In the 1960's, Tutte \cite{tutte_triangle} has enumerated several
kinds of rooted planar maps, obtaining elegant formulas. Among
these, planar rooted triangulations are counted by the formula
\begin{equation}
  \label{tutte_tamarinte}
  2 \frac{(4 n+1)!}{(n+1)! (3 n+2)!}.
\end{equation}
This formula can be written as
\begin{equation*}
  \frac{(4n+1)\dots(3n+3)}{(3)\dots(n+1)},
\end{equation*}
hence comes from the weak Weight
\begin{equation}
  \label{tamarinte_weight}
  (3,\dots,n+1;4n+4).
\end{equation}
Once again, it is not clear if the $q$-Milnor number is a
polynomial with positive coefficients. This property can be checked
for $n \leq 60$, but there is no known combinatorial statistics to
explain this property.

For small $n$, the Weights in this family are
$\dy{A}_1,\dy{A}_3,W_{13},\dy{A}_4 \times W_{17}$.


Assuming that \eqref{tamarinte_weight} always defines a Weight, one can look for posets satisfying the Coxeter
criterion. Besides triangulations, there are now several other
families of combinatorial objects counted by formula
\eqref{tutte_tamarinte}: the set of all intervals in Tamari lattices
\cite{chapoton_tamarintes}, extended fighting fishes
\cite{duchi_henriet}, etc.

So far, no sequence of partial orders with the correct Coxeter
polynomials has been found.
The most natural partial order on Tamari intervals, involved in the
relation with the diagonals of the associahedra, is defined by
$[a,b] \leq [a',b']$ if and only if $a \leq a'$ and $b \leq b'$. It
does not meet the Coxeter criterion.  
The naive partial order by inclusion of intervals does not work either. 

\section{Green mutation poset for the cyclic quivers}

In the theory of cluster algebras, one can associate mutations graphs
to quivers. Using the notion of green mutations, one can define an
orientation of the mutation graph. When the mutation graph is finite,
one obtains finite posets, among which the Tamari lattice considered
in section \ref{catalan}. For details, see the survey \cite{keller_green}.

Let us consider here the sequence of posets defined in this way,
starting from the cyclic quivers on $n$ vertices, with $n\geq 2$. In
the classification by Fomin and Zelevinsky of quivers of finite type \cite{fomin_zelevinsky},
these quivers have type $\dy{D}_n$. The number of elements in the
posets (clusters) is therefore given by the Coxeter-Catalan number of
type $\dy{D}_n$, which is
\begin{equation}
  (3n-2) c_{n-1},
\end{equation}
where $c_n$ is the Catalan number \eqref{cat}. This can be written as
the Milnor number of the weak Weight
\begin{equation}
 (4, 2 n, [6, 9, 12, \dots, 3n-3] ; 6 n),
\end{equation}
where the subsequence of degrees inside the bracket is an arithmetic
progression of step $3$.

One can check that this indeed defines Weights for $n\leq 50$.

For small $n \geq 2$, the Weights in this family are
$\dy{A}_2 \times \dy{A}_2,\dy{A}_2 \times \dy{E}_7,\dy{A}_{2}\times\dy{A}_5\times\dy{D}_5,\dy{E}_{13} \times S_{1,0}$.


Experimentally, the green-mutation partial orders for the cyclic
quivers satisfy the Coxeter criterion for the Weight given above. One
therefore expects them to be fractional Calabi-Yau with the prescribed
dimension.

\begin{remark}
  Let $\Cat_n$ be the Catalan Weight introduced in
  section \ref{catalan}.  One can note that the Weight associated
  above to the green mutation poset for the cyclic quiver of type
  $D_n$ is a multiple of $\Cat_{n-1}$. A similar phenomenon seems to
  happen for the sub-poset consisting of positive clusters, namely
  those not meeting the initial cluster, for the Weight $\dy{A}_{n-1} \times \Cat_{n-1}$.
\end{remark}


\appendix

\section{Derived categories of posets and Coxeter polynomials}

\label{cox}

Let $(P, \leq)$ be a finite partial order. One can define the
incidence algebra of $P$ over a field, and consider the category of
finite dimensional modules over this algebra and its bounded derived
category $\deri(P)$. The category $\deri(P)$ has finite global
dimension and possesses Serre and Auslander-Reiten functors.

On the triangulated category $\deri(P)$, the Auslander-Reiten
translation functor $\tau$ is an auto-equivalence. It induces a linear
map on the Grothendieck group $K_0(\deri(P))$, which is a free abelian
group of rank $|P|$. The matrix of this linear map in the basis made
of classes of simple modules can be described as follows.

Pick any total order on $P$ which is an extension of the partial order
$\leq$. Let $L_P$ be the triangular matrix with coefficient $1$ in
position $(i,j)$ if $i \leq j$ and $0$ elsewhere. The Coxeter matrix
$C_P$ is then $-L_P L_P^{-t}$, where $L_P^{-t}$ is the transpose of
the inverse of $L_P$. The Coxeter polynomial is the characteristic
polynomial of the Coxeter matrix.

The Coxeter polynomial of a poset is concretely available in several
computer algebra systems.

\section{Tables and names}

\label{tables}

Here are small tables of named Weights, some of which have appeared in
the article to describe the first few Weights in the families. The
names come either from quivers, root systems or singularity theory.

For most of these Weights, one can find at least one poset whose
derived category should be equivalent to the geometric category
associated with the Weight.

\begin{center}
\begin{tabular}{|l|c|}
  \hline
  \multicolumn{2}{|c|}{Dynkin quivers} \\
\hline
$\dy{A}_n$ & $(1 ; n+1)$ \\
\hline
$\dy{D}_n$ & $(2, n-2 ; 2n-2)$ \\
\hline
$\dy{E}_6 = \dy{A}_2 \times \dy{A}_3$ & $(3,4;12)$\\
\hline
$\dy{E}_7$ & $(2, 3 ; 9)$ \\
\hline
$\dy{E}_8 = \dy{A}_2 \times \dy{A}_4$ & $(3,5;15)$\\
\hline
\end{tabular}
\quad
\begin{tabular}{|l|c|}
  \hline
 \multicolumn{2}{|c|}{Elliptic root systems} \\
\hline
$\dy{E}_6^{(1,1)} = \dy{A}_2 \times \dy{A}_2 \times \dy{A}_2$ & $(1,1,1;3)$ \\
\hline
$\dy{E}_7^{(1,1)} = \dy{A}_3 \times \dy{A}_3$&$(1,1;4)$ \\
\hline
$\dy{E}_8^{(1,1)} = \dy{A}_2 \times \dy{A}_5$&$(1,2;6)$ \\
\hline
\end{tabular}
\end{center}



\smallskip

\begin{center}
\begin{tabular}{|l|c|}
  \hline
 \multicolumn{2}{|c|}{Arnold's unimodal singularities} \\
\hline
$E_{12} = \dy{A}_2 \times \dy{A}_6 $ & $(3, 7;21)$\\
\hline
$E_{13}$ & $(2, 5;15)$\\
\hline
$E_{14} =\dy{A}_2 \times \dy{A}_7 $& $(3, 8;24) $\\
\hline
$Z_{11}$ & $(3, 4;15)$\\
\hline
$Z_{12}$ & $(2, 3;11)$ \\
\hline
$Z_{13}$ & $(3, 5;18)$\\
\hline
$Q_{10} = \dy{A}_2 \times \dy{D}_5$ & $(6, 8, 9;24)$\\
\hline
$Q_{11}$ & $(4, 6, 7;18)$\\
\hline
$Q_{12}=\dy{A}_2 \times \dy{D}_6$ & $(3, 5, 6;15) $\\
\hline
$W_{12}=\dy{A}_3 \times \dy{A}_4$ & $(4, 5;20) $\\
\hline
$W_{13}$ & $(3, 4;16)$ \\
\hline
$S_{11}$ & $(4, 5, 6;16)$\\
\hline
$S_{12}$ & $(3, 4, 5;13)$ \\
\hline
$U_{12} = \dy{A}_2 \times \dy{A}_2 \times \dy{A}_3$ & $(3, 4, 4;12)$  \\
\hline
\end{tabular}
\quad
\begin{tabular}{|l|c|}
  \hline
 \multicolumn{2}{|c|}{Arnold's bimodal singularities} \\
\hline
$E_{18}=\dy{A}_2 \times \dy{A}_9$ & $(3, 10;30) $\\
\hline
$E_{19}$ & $(2, 7;21)$\\
\hline
$E_{20}=\dy{A}_2 \times \dy{A}_{10}$ & $(3, 11;33) $\\
\hline
$Z_{17}$ & $(3, 7;24)$\\
\hline
$Z_{18}$ & $(2, 5;17)$\\
\hline
$Z_{19}$ & $(3, 8;27)$\\
\hline
$Q_{16}=\dy{A}_2 \times \dy{D}_8$ & $(3, 7, 9;21) $\\
\hline
$Q_{17}$ & $(4, 10, 13;30)$\\
\hline
$Q_{18}=\dy{A}_2 \times \dy{D}_9$ & $(6, 16, 21;48) $\\
\hline
$W_{17}$ & $(3, 5;20)$\\
\hline
$W_{18}=\dy{A}_3 \times \dy{A}_6$ & $(4, 7;28) $\\
\hline
$S_{16}$ & $(3, 5, 7;17)$\\
\hline
$S_{17}$ & $(4, 7, 10;24)$\\
\hline
$U_{16}=\dy{A}_2 \times \dy{A}_2 \times \dy{A}_4$ & $(3, 5, 5;15)$\\
\hline
\end{tabular}
\end{center}

\smallskip

\begin{center}
\begin{tabular}{|l|c|}
  \hline
  \multicolumn{2}{|c|}{Quadrilateral singularities} \\
  \hline
  $J_{3,0} = \dy{A}_2 \times \dy{A}_8$ & $(1, 3;9)$ \\
  \hline
  $Z_{1,0}$ &  $(1, 2;7) $\\
  \hline
  $Q_{2,0} = \dy{A}_2 \times \dy{D}_7$ &  $(2, 4, 5;12)$\\
  \hline
  $W_{1,0} = \dy{A}_3 \times \dy{A}_5 $&  $(2, 3;12)$\\
  \hline
  $S_{1,0}$ &  $(2, 3, 4;10)$\\ 
  \hline
  $U_{1,0} = \dy{A}_2 \times \dy{E}_7 $ & $(2, 3, 3;9)$\\
  \hline
\end{tabular}
\end{center}



\bibliographystyle{alpha}
\bibliography{poids}

\end{document}